\documentclass{amsart}
\usepackage{amssymb}
\usepackage{amsmath}
\usepackage{amsfonts}
\usepackage{geometry}

\newtheorem{theorem}{Theorem}
\theoremstyle{plain}
\newtheorem{acknowledgement}{Acknowledgement}

\newtheorem{corollary}{Corollary}

\newtheorem{definition}{Definition}
\newtheorem{example}{Example}

\newtheorem{lemma}{Lemma}

\newtheorem{proposition}{Proposition}
\newtheorem{remark}{Remark}

\numberwithin{equation}{section}
\input{tcilatex}

\begin{document}
\title{On the existence of SLE trace: finite energy drivers and non-constant $\kappa$}
\author{Peter K. Friz, Atul Shekhar}
\address{TU\ and WIAS\ Berlin (first and corresponding author,
friz@math.tu-berlin.de), TU\ Berlin (second author)}

\begin{abstract}
Existence of Loewner trace is revisited. We identify finite energy paths (the ``skeleton of Wiener measure") as natural class of regular drivers for which we find simple and natural
estimates in  terms of their (Cameron--Martin) norm. Secondly, now dealing with potentially rough drivers, a representation of the derivative of the (inverse of the) Loewner flow is given
in terms of a rough- and then pathwise F\"ollmer integral. Assuming the driver within a class of It\^o-processes, an exponential martingale argument implies existence of trace. In contrast
to classical (exact) SLE computations, our arguments are well adapted to perturbations, such as non-constant $\kappa$ (assuming $<2$ for technical reasons) and additional finite-energy drift terms. 
\end{abstract}

\maketitle

\section{Introduction}

Classical theory of Loewner evolution gives a one-to-one correspondence between scalar continuous drivers (no smoothness assumptions) and families of continuously growing compact sets in the complex upper half-plane $\mathbb{H}$. There is much interest in the case where these sets admit a continuous trace (or even better: are given by a simple curve in $\mathbb{H}$).
The famous {\it Rohde--Schramm theorem} \cite{BasicSLE}
asserts that Brownian motion with diffusivity $\kappa \neq 8$ a.s. gives rise to a continuous trace (simple when $\kappa \le 4$), better known as SLE($\kappa$)-curves.\footnote{We shall make no attempt here to review the fundemental importance of SLE theory within probability and statistical mechanics. See e.g. \cite{lawlerbook} and the references therein.} The trace also exists for SLE$(8)$ but the proof follows indirectly from the convergence of uniform spanning tree to SLE$(8)$. (We note that the proofs are probabilistic in nature -- ultimately an application of the Borel-Cantelli lemma -- and gives little insight about the exceptional set.)
Deterministic aspects were subsequently explored by Marshall, Rohde, Lind, Huy Tran, Johannson Viklund and
others (see e.g. \cite{Joh15} and the references therein). We observe a number of similarities between the {\it It\^o map}, which
takes a Brownian driver to a diffusion path) with the {\it Schramm--L\"owner map} which takes a Brownian driver to SLE trace $\gamma$ (a ``rough", in the sense of non-smooth, path in $\mathbb{H}$),
$$
      \Phi_{SL}: \sqrt{\kappa} B (\omega) \mapsto \gamma (\omega).
$$      
In both cases, there is a ``Young" regime (case of drivers with H\"older exponent better than $1/2$) in which case one can fully rely on deterministic theory.\footnote{The analogy is not perfect: Young differential equations of form $dY=f(Y)dX$ are
invariant under reparametrization, hence most naturally formulated in a $p$-variation, $p<2$, context, whereas Loewner evolution is classically tied to parametrization by half-plane capacity. } 
%
%
%
%
%
Also, in both the cases, Brownian motion does not fall in the afore-mentioned ``Young" regime, and yet both the It\^o- resp. (Schramm-)L\"owner map are well-defined measurable maps. While the It\^o map, also thanks to rough path theory,
 is now very well understood, the afore-mentioned proof of the Rohde-Schramm theorem - despite being state of art - is not fully satisfactory. For instance, the case $\kappa=8$ still resists a direct analysis.
Even robustness in the parameter $\kappa$ turns out to be a decisively non-trivial issue, only recently settled in \cite{MR3229999} under the (technical) restriction of $\kappa < 8(2-\sqrt{3}) $. 

\vspace{0.2cm} 
To some extent, the ``pathwise" theory of Loewner evolutions, concerning existence of trace, has been settled by Rohde--Schramm in form of the following

\begin{theorem}\label{RS}  \cite{BasicSLE} Loewner evolution with driver $U$, a continuous real valued path with $U_0 = 0$, admits a continuous trace if and only if \[ \gamma_t := \lim\limits_{y \rightarrow 0+} f_t(iy + U_t)\]exists and is continuous in $t$. In this case, $\gamma$ is the trace. 
\end{theorem}  
In the above theorem, following standard notation, $f_t(z) = g_t^{-1}(z)$ and for each $z \in \bar{\mathbb{H}}\setminus 0$, let $g_t(z)$ denote the solution of the LDE (Loewner's differential equation) \begin{equation}\label{LDE} \dot{g}_t(z) = \frac{2}{g_t(z)- U_t}, \hspace{2mm} g_0(z)= z. \end{equation}

Evenso, it is a non-trivial matter, and the essence of the afore-mentioned Rohde--Schramm theorem, to see that this applies to a.e. Brownian sample path. It is here that one has to work with Whitney-type boxes  and a subtle Borel-Cantelli argument which in fact misses the case $\kappa=8$, subsequently handled with different methods. Readers familiar with the details of the proof may observe that a harmless finite-energy perturbation of the driver will already cause some serious complications, whereas it is a priori clear from the Cameron--Martin theoerem that SLE driven by $\sqrt{\kappa} B +h$, where $h$ is $\int_0^.$ of some $L^2$-function, does produce a continuous trace. (Such perturbations are relatively harmless for the It\^o-map, essentially because integration against $dh=\dot{h}dt$ is deterministic and SDEs driven by $B+h$ can be dealt with via flow decompositions.) 

In fact, we observed with some surprise that Loewner evolution driven by finite-energy paths, despite being the ``skeleton" of Wiener-measure, has not been analyzed. 
With regard to Lind's ``$1/2$-H\"older norm $< 4$" condition, we note that a finite-energy path $h$ is indeed $1/2$-H\"older (by a simple Sobolev embedding), 
but may have arbitrarily large $1/2$-H\"older norm. Evenso, there is a ``poor man's argument" that allows to see that such $h$ generate a simple curve trace:
the remark is that $h$ is {\it vanishing $1/2$-H\"older} in the sense $\sup_{s,t:|t-s|< \varepsilon} \frac{|h_t - h_s|}{|t-s|^{1/2}} \to 0$ with $\varepsilon \to 0$, so that $\gamma$
 is given by suitable concatenation of (conformally deformed) simple curves. A better understanding of the situation is given by Theorem \ref{thm:CM} below. To state it define 
 $\mathcal{H}_t$ as the space of absolutely continuous $h:[0,t] \to \mathbb{R}$, with finite energy i.e. square-integrable derivative $\dot{h}$, and norm-square
 $$
      || h ||^2_t  := \int_0^t |\dot{h_s}|^2 ds.
 $$
Also, $C_T^\alpha$ is the space of paths defined $[0,T]$, with H\"older exponent $\alpha \in (0,1]$.
Write $C_T^{p-var}$ for the space of continuous paths on $[0,T]$ of finite $p$-variation; note that $\alpha$-H\"older imflies finite $1/\alpha$-variation.  At last, $C^{\dots}_{0,T}$ indicates paths on $[0,T]$ which are started at $0$. For instance, given a standard Brownian motion $B$, with probability one $\sqrt{\kappa} B (\omega) \in 
C^\alpha_{0,T}$ for any $\alpha < 1/2$.

\pagebreak
 
\begin{theorem} \label{thm:CM}  
Let $T>0$ and $U \in \mathcal{H}_T$.

\noindent (i) The following estimate holds for all $y>0$ and $t \in [0,T]$, 
\begin{equation}\label{eq:CM} |f_t'(iy+U_t)| \leq \exp \biggl[\frac{1}{4} || U ||_t^2\biggr]. \end{equation}     
\noindent  (ii) The Loewner-trace $\gamma=:\Phi_{SL} (U)$ exists and is a simple curve.

\noindent (iii) The trace is uniformly $1/2$-H\"older in the sense that, for some constant $C$,

\begin{equation}  || \gamma ||_{1/2} \le C \exp \biggl[ C ||U||_T^2 \biggr]  \end{equation}    

\noindent (iv) The map $t \mapsto \gamma(t^2)$ is Lipschitz continuous on $[0,T]$. As a consequence, the trace is of bounded variation and Lipschitz away from $0+$.


\noindent (v) On bounded sets in $\mathcal{H}_T$, the Schramm--Loewner map is continuous from $C[0,T]$ to $C^{1/2-\epsilon}([0,T],\bar{\mathbb{H}})$, any $\epsilon >0$.
  
\noindent (vi) The Schramm--Loewner map is continuous from $\mathcal{H}_T$ to $C^{(1+\epsilon)-var}$, any $\epsilon > 0$. 

%

\end{theorem}

Our second contribution is a {\it pathwise} inequality that is well-suited to obtain existence of trace for stochastic drivers beyond Brownian motion. To state it, let us say that $U: [0,T] \to \mathbb{R}$ has {\it finite quadratic-variation in sense of F\"ollmer} if (along some {\it fixed} sequence of partitions $\pi=(\pi_n)$ of $[0,T]$, with mesh-size going to zero)
$$
\exists \lim_{n\to\infty} \sum_{ [r,s] \in \pi_n } (U_{s \wedge t} -  U_{r \wedge t})^2 =: [U]^\pi_t
$$ 
and defines a continuous map $t \mapsto [U]^\pi_t \equiv [U]_t$. A function $V$ on $[0,T]$ is called {\it F\"ollmer-It\^o integrable} (against $U$, along $\pi$) if
$$
\exists \lim_{n\to\infty} \sum_{ [r,s] \in \pi_n } V_s (U_{s \wedge t} -  U_{r \wedge t}) =: \int_0^t V d^\pi U.
$$
(F\"ollmer \cite{Foe81} shows that integrands of gradient form are integrable in this sense and so defines pathwise integrals of the form $\int \nabla F(U) d^\pi U$.)  If the bracket is furthermore Lipschitz, in the sense that
\begin{equation} \label{LipFB}
       \sup_{0 \le s < t \le T}  \frac{[U]_t - [U]_s}{t-s} \le   \kappa < \infty,
\end{equation}
write $U \in \mathcal{Q}_T^{\pi,\kappa}$. For instance, whenever $\pi$ is nested, a martingale argument shows that $\sqrt{\kappa} B (\omega) \in \mathcal{Q}_T^{\pi, \kappa}$ with probability one. We insist again that the following
result is entirely deterministic and highlights the role of the (pathwise) property (\ref{LipFB}) relative to existence of SLE($\kappa$) trace.

\begin{theorem} \label{thm:main} Let $T>0$ and $U \in C^\alpha_{0,T} \cap  \mathcal{Q}_T^{\pi,\kappa}$ for some $1/3 < \alpha < 1/2$ and $\kappa < 2$.
For fixed $t \in [0,T]$ set $\beta_s  := U_t - U_{t-s}$ 
and then, for arbitrarily chosen $A \in \mathcal{H}_t$, consider the decomposition\footnote{One could write  $\beta^{(t)} = N^{(t)} + A^{(t)}$ to emphasize dependence on $t$.}
$$
       \beta = N + A.
$$
Then there exists a continuous function $\dot{G}$, F\"ollmer integrable against $N$, so that for some $b>2, p>1$ and $\epsilon > 0$, 
depending only on $\kappa$, we have,  for all $y>0$ and $t \in [0,T]$,  
\begin{equation}  \label{key1}
 |f_t'(iy + U_t)|^b  \leq   \exp\biggl( b\int_0^t \dot{G}_rd^\pi N_r - \frac{pb^2}{2}\int_0^t \dot{G}_r^2d[N]^\pi_r\biggr) \exp\biggl( \frac{b}{4\epsilon} || A ||_t^2 \biggr).
 \end{equation}
\end{theorem}

Several remarks are in order.

\begin{remark} $U \in  \mathcal{Q}_T^{\pi,\kappa}$ iff $N \in  \mathcal{Q}_T^{\pi,\kappa}$ since $[N]_t = [\beta]_t = [U]_T - [U]_{T-t}$. 
\end{remark} 
\begin{remark}  An explicit form of $\dot G$ is found in (\ref{eq:Gdot}). Remark that $\dot G_s$ is obtained as function of $(\beta_u: 0 \le u \le s)$, and in fact is controlled by $\beta$ in the sense of Gubinelli \cite{Gub04} or \cite[Ch. 4]{FH14}, which is a technical aspect in the proof. 
\end{remark} 
\begin{remark} We believe the restriction $\kappa < 2$ to be of technical nature.
\end{remark} 

Write $C^{w,1/2}_{0,T}$ for ``weakly" 1/2-H\"older paths on $[0,T]$, started at zero. Following \cite{JVL11}, this means a modulus of continuity of the form $\omega(r) = r^{1/2} \varphi(1/r)$ for a ``subpower" function $\varphi$ (that is, $\varphi(x) = o(x^\nu)$ for all $\nu>0$, as $x\to \infty$). Thanks to L\'evy's modulus of continuity, with probability one, $\sqrt{\kappa} B (\omega) \in C^{{w, 1/2}}_{0,T} \subset C^\alpha_{0,T}$ for any $\alpha < 1/2$. (A general Besov--L\'evy modulus embedding appears e.g. in \cite[p.576]{FV10}.) 

\begin{corollary} \label{cor:1} Let $T>0$ and consider random $U=U(\omega)$ with $U(\omega) \in C^{w,1/2}_{0,T} \cap  \mathcal{Q}_T^\kappa$ for $\kappa < 2$ a.s.
For fixed $t \in [0,T]$, define $\beta$ as before and assume $\beta$ is a continuous semimartingale w.r.t. to some filtration, with canonical decomposition
 $\beta = N + A$ into local martingale $N$ and bounded variation part $A \in \mathcal{H}_t$, so that $ || A ||_t^2 $ has sufficiently high (depending only on $\kappa$) exponential moments finite uniform in $t$. Then the Loewner-trace $\gamma  =: \Phi_{SL} (U)$ exists.
\end{corollary} 

\begin{proof} By assumption, we can apply Theorem \ref{thm:main} to a fixed realization of $U=U(\omega)$ in a set of full measure. Moreover, in view of the assumed semimartingale structure of $\beta = N +A$, our interpretation of the right-hand side of (\ref{key1}) can now be in classical It\^o-sense. (In the semimartingale case, 
the F\"ollmer's integral and quadratic variation, along suitable sequences of partitions, coincide with It\^o's notion.)

With $b>2$ and then $p>1,\epsilon>0$ as in Theorem \ref{thm:main} , let $q$ be the H\"older conjugate of $p$. H\"older's inequality gives
$$
\mathbb{E} [ |f_t'(iy + U_t)|^b]  \le \mathbb{E}\biggl[  \mathcal{E}  \biggl( pb\int_0^t \dot{G}dN \biggr)    \biggr]^{\frac{1}{p}} \mathbb{E} \biggl[\exp\biggl( \frac{qb}{4\epsilon} \int_0^t \dot{A}_r^2dr\biggr)\biggr]^{\frac{1}{q}}
$$
where $\mathcal{E}  (...)$ denotes the stochastic exponential. Since $\dot{G}$ is adapted to $\beta$, its integral against $N$, is again a local martingale and so it the stochastic exponential. By positivity it is also a super-martingale, started at $1$, and thus of expectation less equal one.
Hence, for $b>2$, have $\mathbb{E} [ |f_t'(iy + U_t)|^b]  < \infty$, uniformly in $t \in [0,T]$ and $y>0$. Together with $U \in C^{w,1/2}_{0,T}$ a.s. this is well-known (cf. appendix) to imply existence of trace.
\end{proof} 

Again, some remarks are in order.

\begin{remark} We cannot make a semimartingale asumption for the Loewner driver $U$ since the time-reversal of a semimartingales can fail to be a semimartingale. That said, time-reversal of diffusion was studied by a  number of authors including Millet, Nualart, Sanz, Pardoux ... and sufficient conditions on ``diffusion Loewner drivers" could be given by tapping into this literature. 
\end{remark} 

%

\begin{remark} As revealed by the above proof of the corollary,  the only purpose of the semimartingale assumption of $\beta$ is to get good concentration of measure for $\int \dot{G} d N$. Recent progress on concentration of measure for pathwise stochastic integrals \cite[Ch.11.2]{FH14}, also {\rm [Ch. 5]} for some Gaussian examples of finite QV in F\"ollmer sense, suggest a possibility to study random Loewner evolutions without martingale methods.

\end{remark} 

Theorem \ref{thm:main} and its corollary have little new to say about existence of trace for SLE$_{\kappa}$, especially with its restriction $\kappa < 2$. However, it is capable of dealing with non-Brownian drivers, including situations with non-constant $\kappa$, and $\mathcal{H}$-perturbations thereof. 

\begin{example} {\bf (Classical SLE$_\kappa$)} As a warmup, consider Loewner driver $U = \sqrt{\kappa} B$ with fixed $\kappa < 2$. Since, for fixed $t$, $\beta_s := U_t - U_{t-s}$ defines another Brownian motion, we can trivially decompose with $N = \beta, A \equiv 0$ and thus obtain a.s. existence of trace for SLE$_\kappa$ immediately from the above corollary.
\end{example}

\begin{example} {\bf (non-constant $\kappa$)} Consider measurable $\kappa: [0,T] \to [0,\bar\kappa]$, with $\bar\kappa<2$, and then
$$ 
     U_t = \int_0^t \sqrt{\kappa(s)} dB_s.
$$
A.s. existence of trace for ``SLE$_\kappa$ with non-constant $\kappa$" follows immediately from the above corollary. Remark that for piecewise constant $\kappa$, given by $(\kappa_i)$ on a finite partition of $[0,T]$, this conclusion can also by given by a suitable concatenation argument, relying on a.s. existence of trace for each classical SLE$_{\kappa_i}$.
\end{example}

\begin{example} {\bf ($\mathcal{H}$-perturbations)} Consider, with $h \in \mathcal{H}_T$,
$$ 
     U_t = \sqrt{\kappa} B_t + h_t 
$$ 
Then  $\beta_s :=\sqrt{\kappa} (B_t - B_{t-s}) + h_t  - h_{t-s}$. The corollary applies with Brownian motion, $N_t = \sqrt{\kappa}( B_t - B_{t-s} )$, and deterministic $A_t = h_t  - h_{t-s}$. Remark that a.s. existence of trace, for $\kappa >0$, is also obtained as consequence of existence of trace for classical SLE$_{\kappa}$ and the Cameron--Martin theorem. Modifying the example to 
$$ 
     U_t = \int_0^t \sqrt{\kappa(s)} dB_s + h_t,
$$
without imposing a lower positive bound on $\kappa$, rules out the Cameron--Martin argument, but Corollary \ref{cor:1} still applies and yields existence of trace a.s.
\end{example}
\begin{example} {\bf (Ornstein--Uhlenbeck drivers)} Consider $U_t = Z_t - Z_0$ where $Z$ is a standard OU process, say with dynamics $dZ = - \lambda Z dt + \sqrt{\lambda} dB_t$  started in its invariant measure. By reversibility of this process, the time-reversed driver $\beta$ has the same law. Existence of trace (for SLE driven by such OU processes) is then a consequence of Corollary \ref{cor:1}.
 \end{example}
 
\begin{corollary}\label{F(B)} Consider 
$$
       U_t = F(t, B_t).
$$ 
with 
$$  F=F(t,x) \in C^{1,2}, F(0,0)=0 \text{ and  } |F'(t,x)|^2 \leq {\kappa} < 2.$$
 Assume furthermore, for $\alpha $ large enough (depending only on $\kappa$)
\begin{equation} \label{momentofF}  \mathbb{E}\biggl[ \exp\biggl( \alpha \int_0^T \biggl\{\dot{F}(r, B_{r}) - \frac{1}{2}F''(r, B_{r}) +  F'(r, B_{r})\frac{B_{r}}{r} \biggr\}^2dr \biggr)\biggr] < \infty.
\end{equation}
Then the Loewner-trace $\gamma  =: \Phi_{SL} (U)$ exists.
\end{corollary}

%

\begin{example} Fix $p>0$ and consider, for the sake of argument on $[0,T]$ with $T \le 1$, 
$$ U_t = t^p B_t.
$$
We insist that there is no ``cheap" way to such results. In particular, there is no ``comparison result" for SLE that would yield existence of trace based on $t^p \le 1$ on $[0,1]$ and existence of trace for SLE$_1$, say. (A related question by O. Angel was negatively answered in \cite{LMR10}.)

This is a special case of $U_t = F(t,B_t)$. To apply Corollary \ref{F(B)} one needs to check condition (\ref{momentofF}) which boils down to exponential moments for 
$$
       Z_T := \alpha \int_0^T \{ r^{p-1} B_r \}^2 dr. 
$$
with fixed large $\alpha$, depending on $\kappa$. Note that $\mathbb{E} (Z_T) < \infty$. The centered random variable 
$ Z^c_T := Z_T - \mathbb{E} (Z_T)$ lives in the second homogenous chaos over Wiener-space. Exponential moments
of $Z^c_T$ are then guaranteed, e.g. by using results from \cite[Ch. 5]{Ledoux}, provided that the second moment of $Z_T^c$ is small enough. But this can be achieved by
choosing $T>0$ small enough. 
\end{example}

\begin{example} Consider
$$ U_t = t log( 1 + B_t^2).$$
we leave it to the reader to check that Corollary \ref{F(B)} applies and yields existence of trace on $[0,T]$ with $T$ small enough.
\end{example}
\begin{acknowledgement}
A. S. acknowledges support from the Berlin Mathematical School (BMS). P. F. received funding from the European Research Council under the European 
Union's Seventh Framework Programme (FP7/2007-2013) / ERC grant agreement nr. 258237. Both authors would like to thank 
S. Rohde for numerous discussions.
\end{acknowledgement}

\section{Some exact presentations of $f'$.}

\begin{lemma}
For each fixed $t\geq 0$ and $U \in C[0,t]$, define $ \beta_s = U_t - U_{t-s}$, $0 \le s \le t$. Then 
\begin{equation}\label{crucialformula}\log|f_t'(z+U_t)| = \int_0^t \frac{2(X_r^2 - Y_r^2)}{(X_r^2 + Y_r^2)^2}dr
\end{equation}
where $z = x + iy$ and $(X_s, Y_s), s \in [0,t]$ is the solution of the ODE 
\begin{align}\label{XYODE} dX_s & = d\beta_s - \frac{2X_s}{X_s^2 + Y_s^2} ds , \hspace{2mm} X_0 = x  \\ 
dY_s &= \frac{2Y_s}{X_s^2 + Y_s^2}ds, \hspace{2mm} Y_0 = y .
\end{align}
Moreover, $G_s := \beta_s-X_s$ defines a $C^1$-function with,
 \begin{equation} \label{eq:Gdot} 
  \dot{G}_s  = \frac{2X_s}{X_s^2 + Y_s^2}.
\end{equation}

\end{lemma}

\begin{proof}For each $ z \in \mathbb{H}$, the path $ g_{t-s}(f_t(z))$ joins $z$ to $f_t(z)$ as $s $ varies from $0$ to $t$. It is then easy to see that  \[f_t(z + U_t) = P_t(z) + U_t\] where $P_s(z)$ for $ s \in [0, t]$ is the solution of ODE 
\[\dot{P}_s(z) = \frac{-2}{P_s(z) + \beta_s},  \hspace{2mm} P_0(z) = z \]
Writing in polar form, $ P_s' = r_se^{i\theta_s} $, we see that 
 \[ Re(\frac{|P_s'|}{P_s'} \partial_s P_s' ) = Re(e^{-i\theta_s}( e^{i\theta_s} \partial_s r_s + i r_s e^{i\theta_s}\partial_s \theta_s ) ) = \partial_s r_s \] 
 
 So it follows that, \[ \partial_s log|P_s'| = Re( \frac{1}{P_s'} \partial_s P_s') \] 

 Noting that $ \partial_s P_s' = (\partial_s P_s)'  $,
 
 \[ \partial_s log|P_s'| = Re( \frac{1}{P_s'} (\frac{-2}{ P_s + \beta _s } )' ) =  2Re( (P_s + \beta_s)^{-2} )\] 
 
 \[ \implies log|P_s'| = 2 \int_0^s Re( (P_r + \beta_r)^{-2}  )dr \]
 and the claim follows. \\

\end{proof}

\begin{proposition}\label{singular-nonsingular}  Fix $t\geq 0$. Let $U \in C^\alpha$ with $\alpha \in (1/3,1/2]$. With $G,\beta$ as in the previous lemma,
\begin{equation} 
\log |f_t'( z +U_t)| = M_t -  \int_0^t \dot{G}_r^2 dr  + \log(\frac{Y_t}{y})  -\log (\frac{X_t^2 + Y_t^2}{ x^2 + y^2})   \label{e:13}
\end{equation}
where 
$M_t$ is given as rough integral
\begin{equation} \label{equ:MRP}   M_t = \lim_n \sum_{[s,t] \in \pi_n} \dot{G}_s (\beta_t - \beta_s)  + \dot{G}'_s \frac{1}{2}(\beta_t - \beta_s)^2   
\end{equation}
with the Gubinelli derivate $$\dot{G}'_s := {\dot{Y}_s}/{Y_s} - \dot{G}_s^2  .$$
%
%
%
%
If in addition, $U$ (equivalently: $\beta$, as defined in the previous lemma) has continuous finite quadratic-variation in sense of F\"ollmer (along $\pi$) then
\begin{equation} 
\log |f_t'( z +U_t)| = M^\pi_t + \frac{1}{2} \int_0^t \dot{G}'_s d[\beta]^\pi_s -  \int_0^t \dot{G}_r^2 dr  + \log(\frac{Y_t}{y})  -\log (\frac{X_t^2 + Y_t^2}{ x^2 + y^2})   \label{e:13}
\end{equation}
with (deterministic) F\"ollmer--It\^o integral
\begin{equation} \label{IFI}
 M^\pi_t = \lim_n \sum_{[u,v] \in \pi_n} \dot{G}_u (\beta_v - \beta_u) =: \int_0^t \dot{G} d^\pi \beta . \end{equation}
 \end{proposition}
 \begin{remark}
 When $U \in \mathcal{H}$, i.e. in the case of finite energy driver, $[\beta] \equiv 0$ and $M \equiv M^\pi$ reduces to a classical Riemann-Stieltjes integral.
 \end{remark}
 
\begin{proof} 
 Consider first the case of $U$ (equivalently: $\beta$) in $C^1$. Then
\begin{align*}
& \qquad \dot{G}_rd\beta_r - \frac{1}{2}\dot{G}_r^2dr + \frac{Y_rdY_r}{X_r^2 + Y_r^2} 
- \frac{2X_rdX_r + 2Y_rdY_r}{X_r^2 + Y_r^2}   \\ & = \frac{2X_r}{X_r^2 + Y_r^2}d\beta_r - \frac{2X_r^2}{(X_r^2 + Y_r^2)^2}dr -\frac{2X_rdX_r}{X_r^2 + Y_r^2} -\frac{Y_rdY_r}{X_r^2 + Y_r^2} \\ &= \frac{2X_r}{X_r^2 + Y_r^2}d(\beta_r - X_r)
- \frac{2X_r^2}{(X_r^2 + Y_r^2)^2}dr -  \frac{2Y_r^2}{(X_r^2 + Y_r^2)^2}dr \\ &=  \frac{4X_r^2}{(X_r^2 + Y_r^2)^2}dr - \frac{2X_r^2}{(X_r^2 + Y_r^2)^2}dr -  \frac{2Y_r^2}{(X_r^2 + Y_r^2)^2}dr \\ &= \frac{2(X_r^2-Y_r^2)}{(X_r^2 + Y_r^2)^2}dr
\end{align*}

Next note that \[ \frac{1}{2}\dot{G}_r^2dr + \frac{1}{2}\dot{Y}_r^2dr =  \frac{\dot{Y}_r}{Y_r}dr\]and \[\frac{Y_rdY_r}{X_r^2 + Y_r^2} = \frac{1}{2}\dot{Y}_r^2dr =  \frac{\dot{Y}_r}{Y_r}dr - \frac{1}{2}\dot{G}_r^2dr \]
Putting all together, we get 
\[ \frac{2(X_r^2-Y_r^2)}{(X_r^2 + Y_r^2)^2}dr = \dot{G}_rd\beta_r - \dot{G}_r^2dr + \frac{\dot{Y}_r}{Y_r}dr - \frac{2X_rdX_r + 2Y_rdY_r}{X_r^2 + Y_r^2}   \]and integrating both side, the claim follows with $M_t = \int_0^t \dot{G_s} d\beta_s$.
In the case of rough driver, meaning $U$ (equivalently: $\beta$) in $C^\alpha$ with $\alpha > 1/3$ the idea is to exploit a cancellation between 
$ \dot{G}_rd\beta_r $ and $-  \frac{2X_rdX_r}{X_r^2 + Y_r^2}$. We can in fact rely on basic theory of controlled rough path to see existence of the rough integral $M_t$. It suffices
to note that the integrand $\dot{G}$ is controlled by the integrator $\beta$. To see this, write 
 \[ \dot{G}_s  = \frac{2X_s}{X_s^2 + Y_s^2} =: \varphi (X_s, Y_s) \]
 and since $\varphi$ is smooth and well-defined (as long as $y>0$ fixed), and $Y$ plainly Lipschitz, it is easy to see that (or just apply directly Exercise 7.8 in \cite{FH14}) 
 
 \[ \dot{G}_s - \dot{G}_r  = \partial_x \varphi (X_s, Y_s) ({X}_s - {X}_r) + O (|s-r|^{2\alpha}) =\partial_x \varphi (X_s, Y_s) ({\beta}_s - {\beta}_r) + O (|s-r|^{2\alpha}) \]
 
 which guarantees existence of (\ref{equ:MRP}) as rough path integral (Theorem 4.10 in \cite{FH14}). The second part concerning the splitting into It\^o--F\"ollmer integral and quadratic variation part, is similar to  \cite[Ch. 5.3]{FH14}, using in particular {\rm Lemma 5.9.}.
 \end{proof}

\section{A deterministic estimate on $f'$.}

\begin{theorem} \label{thm:keyest} In the context of Proposition \ref{singular-nonsingular}, with continuous finite quadratic-variation in sense of F\"ollmer so that
$  d [\beta]^\pi_s / ds \le \kappa < 2 $ one has the following estimate
\begin{align*}
 |f_t'( iy + U_t)| &  \leq \exp \biggl[  M^\pi_t  -  \int_0^t \dot{G}_r^2   d(r +  \tfrac{1}{2}[\beta]_r) \biggr]
\end{align*}
where $\int_0^t \dot{G}_r d^\pi \beta_r = M^\pi_t$ is the It\^o-F\"ollmer type integral introduced in (\ref{IFI}).
\end{theorem}

\begin{proof}

From (\ref{e:13}) and definition for $G'$, also taking $x=0$ so that $z = iy$ (i.e. $x=0$),


\begin{align*}
\log| f_t'(iy + U_t)| &= \int_0^t \dot{G}_rd\beta_r -\int_0^t \dot{G}_r^2dr - \frac{1}{2}\int_0^t \dot{G}_r^2d[\beta]_r \\ 
& + \log(\frac{Y_t}{y}) - \log(\frac{X_t^2 + Y_t^2}{y^2}) + \frac{1}{2}\int_0^t \frac{\dot{Y}_r}{Y_r}d[\beta]_r
\end{align*}  
Using positivity of $\dot{Y}_r / {Y_r}$,
\begin{align*}\label{dropterm}
&\log(\frac{Y_t}{y}) - \log(\frac{X_t^2 + Y_t^2}{y^2}) + \frac{1}{2}\int_0^t \frac{\dot{Y}_r}{Y_r}d[\beta]_r \\ & \leq \log(\frac{Y_t}{y}) - 2\log(\frac{Y_t}{y}) + \frac{\kappa}{2} \int_0^t \frac{\dot{Y}_r}{Y_r}dr \\ &= (\frac{\kappa}{2}- 1)\log(\frac{Y_t}{y}) \\ & \leq 0.
\end{align*}
and the desired estimate follows. 
\end{proof}

\section{Proof of Theorem \ref{thm:main}}    

As immediate corollary of Theorem \ref{thm:keyest}, noting
\begin{align*}
\int_0^t \dot{G}_r^2 dr \geq \frac{1}{\kappa} \int_0^t \dot{G}_r^2 d[\beta]^\pi_r , 
\end{align*}
we obtain the (still pathwise) estimate

\begin{equation} \label{estLM}
 |f_t'( iy + U_t)|  \leq \exp \biggl[  \int_0^t \dot{G}_rd^\pi \beta_r - \biggl( \frac{1}{2} + \frac{1}{\kappa} \biggr)\int_0^t \dot{G}_r^2d[\beta]^\pi_r \biggr].  
\end{equation}
and then, with $b := 2 (\frac{1}{2} + \frac{1}{\kappa} )$ 
\begin{align*}
 |f_t'( iy + U_t)|^b & \leq \exp \biggl[ b \int_0^t \dot{G}_rd^\pi \beta_r -  \frac{b^2}{2} \int_0^t \dot{G}_r^2d[\beta]^\pi_r \biggr].  
\end{align*}
Note $b>2$, a consequence of $\kappa <2$. We now assume that, for some $A \in \mathcal{H}_t$,

$$
\beta = N + A
$$
It is then immediate, using $[\beta] = [N]$, that
$$ |f_t'( iy + U_t)|  \leq \exp \biggl[  \int_0^t \dot{G}_r dN^\pi_r + (*) - \int_0^t \dot{G}_r^2   d(r +  \tfrac{1}{2}[N]^\pi_r) \biggr] 
$$
where
$$
(*) = \int_0^t \dot{G}_r d A_r = \int_0^t \dot{G}_r\dot{A}_rdr \leq \epsilon \int_0^t \dot{G}_r^2dr + \frac{1}{4\epsilon} \int_0^t \dot{A}_r^2 dr.
$$
As a consequence, arguing exactly like in obtaining (\ref{estLM})
$$ |f_t'( iy + U_t)|  \leq \exp \biggl[  \int_0^t \dot{G}_r dN^\pi_r -   \biggl( \frac{1-\epsilon}{\kappa} + \frac{1}{2} \biggr)   \int_0^t \dot{G}_r^2   d[N]_r) \biggr] . \exp \biggl[\frac{1}{4\varepsilon} || A ||_{t}^2\biggr].
$$

\section{Finite energy drivers, proof of Theorem \ref{thm:CM} }

We show (i) estimate (\ref{eq:CM}), (ii) existence of Loewner-trace $\gamma$ as a simple curve, (iii) uniform $1/2$-H\"older regularity of $\gamma$, (iv) Lipschitz continuity of $ \gamma(t^2)$, (v) continuity of the Schramm--Loewner in uniform topology, on bounded sets in $\mathcal{H}_T$ and (vi) continuity of the trace in $1+ \epsilon$-variation topology w.r.t. Cameron-Martin topology on the driver.  

(i) The proof of estimate (\ref{eq:CM}) is a straight-forward consequence of Theorem \ref{thm:keyest}. 
Indeed, let $U$ of finite energy on $[0,t]$, note that $U$ and $\beta$ (with $\beta_s := U_t - U_{t-s}$ here) have zero quadratic variation. Then 
\begin{align*}
\log|f_t'( iy + U_t)| & \leq M_t -  \int_0^t \dot{G}_r^2 dr
\end{align*}
and conclude with \begin{align*} M_t  = \int_0^t \dot{G}_r d\beta_r & = \int_0^t \dot{G}_r\dot{\beta}_rdr \\   &\leq  \int_0^t \dot{G}_r^2dr + \frac{1}{4} \int_0^t \dot{\beta}_r^2dr. 
\end{align*}

In fact, from Proposition \ref{singular-nonsingular}, since $\beta$ has zero quadratic variation,  
\begin{equation*}
\log |f_t'( z +U_t)| = \int_0^t \dot{G}_rd \beta_r  -  \int_0^t \dot{G}_r^2 dr  + \log(\frac{Y_t}{y})  -\log (\frac{X_t^2 + Y_t^2}{ x^2 + y^2})   \label{e:13}
\end{equation*}

and using the same argument as above, we obtain a better bound  
\begin{equation}\label{finerCMbound}
|f_t'(x + iy + U_t)| \leq \frac{y}{Y_t}\biggl( 1 + \frac{x^2}{y^2}\biggr)\exp\biggl[\frac{1}{4}||U||_t^2\biggr]
\end{equation}
which implies $|f_t'(z + U_t)|$ remains bounded if $z$ remains in a cone $\{z | |Re(z)| \leq M Im(z)\}$

(ii) This is clear from part (i) and Lemma \ref{decaylemma} in the appendix, where it is shown
\[ \gamma_t = \lim\limits_{y \rightarrow 0+ } f_t(iy + U_t)\] exists as a continuous limit. The fact that $\gamma$ is simple follows e.g. from \cite{Lind} or \cite{LIPGRAPHhuy}.

(iii) We show that \[ ||\gamma||_{\frac{1}{2}} \leq g ( ||U||_{T})\] for some continuous function $ g : [0, \infty) \rightarrow (0 , \infty)$. (In fact, in the proof below reveals the possible choice $g(x) = Ce^{Cx^2}$.) %
%
%
%
%
%
%
 Define \[ v(t,y) := \int_0^y |f_t'(ir + U_t)|dr\]Note that, \[|\gamma(t) - f_t(iy + U_t)| \leq v(t,y)\]
and by an application of Koebe's one-quater Theorem, 

\begin{equation}\label{koebe}
 v(t,y) \geq \frac{y}{4} |f_t'(iy + U_t)| 
\end{equation}

In the proof below, we will choose $ y = \sqrt{t-s}$. Now, 
\begin{align*}
|\gamma(t) - \gamma(s)| \leq & |\gamma(t) - f_t(U_t + iy) | \\ & + |\gamma(s) - f_s(U_s + iy) |   \\
&  + |f_t(U_s + iy) - f_s(U_s+ iy)| \\ & + | f_t(U_t + iy) - f_t(U_s + iy)|    
\end{align*}

The first two terms are bounded by $ v(t,y)$ and $v(s,y)$ respectively. For the third term, Lemma $3.5$ in \cite{JVL11} and (\ref{koebe}) implies, 
\[ |f_t(U_s + iy) - f_s(U_s+ iy)| \leq C v(s,y)\]

For the fourth term, \[ | f_t(U_t + iy) - f_t(U_s + iy)|  \leq |U_t - U_s| \sup_{r \in [0,1]}|f_t'(rU_t + (1-r)U_s + iy )|\]

Note that \[|U_t - U_s| \leq y ||U||_{\frac{1}{2}}   \] and by Lemma $3.6 $ in \cite{JVL11}, there exist constant $C$ and $ \alpha $ such that 
\begin{align*}
|f_t'(rU_t + (1-r)U_s + iy )| & \leq C \max\biggl\{1,\biggl(\frac{|U_t - U_s|}{y}\biggr)^\alpha\biggr\} |f_t'(iy + U_t)| \\ & \leq
C \max\biggl\{1, ||U||_{\frac{1}{2}}^{\alpha}\biggr\} |f_t'(iy + U_t)| 
\end{align*} 
and using (\ref{koebe}) again, 
\[ | f_t(U_t + iy) - f_t(U_s + iy)| \ \leq C ||U||_{\frac{1}{2}} \max\biggl\{1, ||U||_{\frac{1}{2}}^\alpha \biggr \} v(t,y)\] 
Finally, from part (i) 
 \[v(t,y) \leq y \exp\biggl\{\frac{1}{4}||U||_{T}^2 \biggr\} \] \[ ||U||_{\frac{1}{2}} \leq ||U||_{T} \] 
giving us 
\[ |\gamma_t - \gamma_s| \leq \sqrt{t-s} g(||U||_{T})\]
completing the proof. \\

(iv) We will use the results from \cite{LIPGRAPHhuy} for the proof of this part. In particular, we recall from Theorem $3.1$ in \cite{LIPGRAPHhuy} that if $||U||_{\frac{1}{2}} < 4$, then there exist a $\sigma, c  > 0$ such that for all $y > 0$,  \begin{equation}\label{huyestimate} \sigma \sqrt{t} \leq Im(f_t(i y + U_t)) \leq \sqrt{y^2 + 4t}  
\end{equation} 
and from Lemma $2.1$ in \cite{LIPGRAPHhuy} \[ |Re(f_t(i y + U_t))| \leq c \sqrt{t} \]
so that trace $\gamma$ lies inside a cone at 0 and $ |f_t(i \sqrt{t} + U_t )| \leq c\sqrt{t}$. We first assume that $||U||_{\frac{1}{2}, [0,T]} <4$. From the proof of part (iii), we have \[ |\gamma_t - \gamma_s |  \lesssim v(t, \sqrt{t-s}) + v(s, \sqrt{t-s}) \]If $ s,t \geq \epsilon$, using bound \ref{finerCMbound} 
 \[ v(t, \sqrt{t-s}) +  v(s, \sqrt{t-s}) \lesssim (\frac{1}{Y_t} + \frac{1}{Y_s}) (t-s) \lesssim \frac{1}{\sqrt{\epsilon}}(t-s)\]which implies $\gamma$ is Lipchitz on $[\epsilon, T]$. For proving $|\gamma_{t^2} - \gamma_{s^2}| \lesssim |t-s|$, note that we can assume $s= 0$, for otherwise we can consider the image of $\gamma$ under conformal map $g_{s^2} - U_{s^2}$ whose derivative of  the inverse $ f'(. + U_{s^2})$ remains bounded in a cone. Finally again using \ref{finerCMbound} and \ref{huyestimate}, 
\begin{align*}
|\gamma_{t^2}| &\leq |\gamma_{t^2} - f_{t^2}(it + U_{t^2})| + |f_{t^2}(it + U_{t^2})| \\ & \lesssim 
v(t^2, t) + t  \\ & \lesssim \frac{t^2}{Y_{t^2}} + t \\ & \lesssim t
\end{align*}

Finally, if $||U||_{\frac{1}{2}, [0,T]} \geq 4$, we split $ [0,T] = \cup_{k =0}^{n-1} [\frac{kT}{n}, \frac{(k+1)T}{n}]$ such that for each $k$, \\$||U||_{\frac{1}{2}, [\frac{kT}{n}, \frac{(k+1)T}{n}]} < 4$. Note again that \[ \gamma[0,T] = \gamma[0,T/n]\cup f_{\frac{T}{n}}g_{\frac{T}{n}}( \gamma[T/n, T])\] 
From above, $\gamma(t^2)$ is Lipchitz on $[0,T/n]$. The chain $ g_{T/n}(\gamma_{T/n + t}) - U_{T/n}, t \in [0, T/n]$ is driven by $ U_{T/n+ t} - U_{T/n}$ and since $f_{T/n}'(. + U_{T/n})$ is bounded (from \ref{finerCMbound} and the fact that trace remains in a cone), $\gamma(t^2)$ is also Lipchitz on $[T/n, 2T/n]$. Iterating this argument then completes the proof.

(v) Consider If $U^n $ is a sequence of Cameron-Martin paths with $||U^n-U||_\infty \rightarrow 0$ and \[ \sup_n ||U^n||_{T} + ||U||_{T} < \infty\] We need to show that  
for any $ \alpha < \frac{1}{2}$, \[||\gamma^n - \gamma||_{\alpha} \rightarrow 0.\]

We have, 
\begin{align*}
|\gamma^n(t) - \gamma(t)| \leq & |\gamma^n(t) - f_t^n(iy + U_t^n)| \\ &+  |f_t(iy + U_t) - \gamma(t)| \\ &+ |f_t^n(iy + U_t^n) - f_t(iy + U_t)|  
\end{align*}
Note that for fixed $y > 0$, \[ \lim\limits_{ n \rightarrow \infty}
|f_t^n(iy + U_t^n) - f_t(iy + U_t)| = 0 \]uniformly in $t $ on compacts. From part (i) 
\begin{align*}
|\gamma^n(t) - f_t^n(iy + U_t^n)| + |f_t(iy + U_t) - \gamma(t)| & \leq v^n(t, y) + v(t,y) \\ & \leq y \biggl( \exp\biggl\{\frac{1}{4}||U^n||_{T}^2 \biggr\} +  \exp\biggl\{\frac{1}{4}||U||_{T}^2 \biggr\} \biggr )
\end{align*} 
Thus, 
 \[ \lim\limits_{y \rightarrow 0+ } |\gamma^n(t) - f_t^n(iy + U_t^n)| + |f_t(iy + U_t) - \gamma(t)| = 0 \]uniformly in $n$ and $t$. Since $y$ can be chosen arbitrarily small, \[ \lim\limits_{n \rightarrow \infty} 
||\gamma^n - \gamma||_{\infty} = 0\]
Finally note that from part (iii) 
 \[\sup_n ||\gamma^n||_{\frac{1}{2}} < \infty \] and standard interpolation argument concludes the proof.

(vi) Let $U^n$ is a sequence with $||U^n - U||_{T} \rightarrow 0$ as $n \to \infty$. From part (v), we have $||\gamma^n - \gamma||_{\infty} \to 0 $. We claim that $\sup_n ||\gamma^n||_{1-var} < \infty$, which together with standard interpolation argument implies $||\gamma^n - \gamma||_{1+ \epsilon-var} \to 0$ as $n \to \infty$. From the proof of part (iv), we see that if $||U||_{\frac{1}{2}} < 4 - \delta $ (and thus $||U^n||_{\frac{1}{2}} < 4 - \delta$ for $n$ large enough), then $\gamma^n(t^2)$ is  Lipchitz uniformly in $n$ and thus $\sup_n ||\gamma^n||_{1-var} < \infty$. \\
If $||U||_{\frac{1}{2}} \geq 4 $, we choose a $m$ large enough and dissect $ [0,T] = \cup_{k=0}^{m-1}[\frac{kT}{m}, \frac{(k+1)T}{m}]$ such that for all $n$ and $ k \leq m-1$, $||U^n||_{\frac{1}{2}, [\frac{kT}{m}, \frac{(k+1)T}{m}]} < 4 - \delta$ 
and similar iteration argument as in proof of part (iv) again implies  $\sup_n ||\gamma^n||_{1-var} < \infty$, completing the proof.

\section{Proof of Corollary \ref{F(B)} }

We consider Loewner drivers of the form $U_t = F(t, B_t)$ where $B$ is a standard Brownian motion. For a fixed time $t > 0$, the process $\beta_s = \beta_s^t = U_{t} - U_{t-s}$ is the time reversal of $U$. Note that $ W_s = B_t - B_{t-s}$ is another Brownian motion and a martingale w.r.t. to its natural completed filtration $\mathcal{F}_s $ satisfying usual hypothesis. We recall the following classical result on expansion of filtration. See \cite[Ch. 6]{Pro90} for detaills.

\begin{theorem}
Brownian motion $W$ remains a semimartingale w.r.t. expanded filtration $\tilde{\mathcal{F}}_s := \mathcal{F}_s \vee \sigma(W_t) = \mathcal{F}_s \vee \sigma(B_t) $. Moreover, \[ W_s = \tilde{W}_s + \int_0^s \frac{W_t - W_r}{t-r} dr \]
where $\tilde{W}$ is a Brownian motion adapted to the filtration $\tilde{\mathcal{F}}$.
\end{theorem} 

We now prove that $\beta_s$ is a semimartingale w.r.t. to filtration $\tilde{\mathcal{F}}$ and provide its explicit decomposition into martingale and bounded variation part. More precisely, we claim
\begin{align*}
\beta_s &=  \int_0^s F'(t-r, B_{t-r})dW_r + \int_0^s \biggl(\dot{F}(t-r, B_{t-r}) - \frac{1}{2}F^{''}(t-r, B_{t-r})\biggr)dr \\ 
&= \int_0^s F'(t-r, B_{t-r})d\tilde{W}_r \\  & \qquad+ \int_0^s \biggl(\dot{F}(t-r, B_{t-r}) - \frac{1}{2}F^{''}(t-r, B_{t-r}) + F'(t-r, B_{t-r})\frac{B_{t-r}}{t-r} \biggr)dr
\end{align*}
To see this, just use It\^o's formula, \[\beta_s = \int_{t-s}^t F'(r, B_r )dB_r + \int_0^s \biggl(\dot{F}(t-r, B_{t-r}) + \frac{1}{2}F^{''}(t-r, B_{t-r})\biggr)dr\]
and note that by computing the difference between forward (It\^o) and backward stochastic integral, \[ \int_{t-s}^t F'(r, B_r )dB_r =   \int_0^s F'(t-r, B_{t-r})dW_r - \int_0^s F^{''}(t-r, B_{t-r}) dr. \] 

The proof of corollary \ref{F(B)} is the completed by application of Theorem \ref{thm:main}.

%
%
%
%
%
%
%
%
\begin{remark}
If function $F'(t,x)$ is not space depedent, e.g. $F(t, x) = t^px $ or $ F(t, x) = \sqrt{\kappa}x $, we can apply the formula \[ \int_{t-s}^t F'(r)dB_r = \int_0^s F'(t-r)dW_r\]Note that right-hand side is indeed a martingale w.r.t. the filtration $\mathcal{F}$ and we do not have to work with expanded filtration $\tilde{\mathcal{F}}$. In this case the canonical decomposition of $\beta$ is given by \[ \beta_s =  \int_0^s F'(t-r)dW_r + \int_0^s \dot{F}(t-r, B_{t-r})dr\]and Theorem \ref{thm:main} again can be applied considering $\beta$ as a semimartingale w.r.t. the filtration $\mathcal{F}$.
\end{remark}

\section*{Appendix}

We collect some variations on familiar results concerning existence of trace via moments of $f'$.

\begin{lemma}\label{decaylemma} Suppose there exist a $\theta < 1$ and $ y_0 > 0$ such that for all $ y \in (0 , y_0] $
\begin{equation}\label{neededbound} \sup_{ t \in [0,T]} |f_t'(iy + U_t )| \leq y^{-\theta} 
\end{equation}
then the trace exists.

\end{lemma}

\begin{proof}Note that for $ y_1 < y_2 < y_0$, 
\[| f_t(iy_2 + u_t) - f_t(iy_1 + U_t) | = | \int_{y_1}^{y_2} f_t'( ir + U_t) dr| \leq  \int_{y_1}^{y_2} r^{-\theta} dr = \frac{1}{1-\theta}(y_2^{1-\theta} - y_1^{1 - \theta} ) \]which implies that $f_t(iy + U_t)$ is Cauchy in $y $ and thus \[ \gamma_t = \lim\limits_{y \rightarrow 0+ } f_t(iy + U_t)\]exists. For continuity of $\gamma$, observe that 
\[ |\gamma_t -  f_t(iy + U_t) | \leq \frac{y^{1-\theta}}{1-\theta} \]
Now, 
\begin{align*}
|\gamma_t - \gamma_s| & \leq |\gamma_t -  f_t(iy + U_t) | + |\gamma_s-  f_s(iy + U_s) | + | f_t(iy + U_t)  - f_s(iy + U_s)|  \\
& \lesssim y^{1 -\theta} + | f_t(iy + U_t)  - f_s(iy + U_s)|  
\end{align*}
It is easy to see that for $ y > 0$, \[ \lim\limits_{s \rightarrow t }  | f_t(iy + U_t)  - f_s(iy + U_s)| = 0\]and since $y$ was arbitrary, this concludes the proof. 
\end{proof}

%
%
%
%

\begin{lemma}\label{momentbound}If $U$ is weakly $\frac{1}{2}$-Holder and there exist constant $ b > 2 $, $ \theta < 1$ and $ C < \infty$ such that for all $t \in [0,T]$ and $y > 0$
 \[ \mathbb{P}[ |f_t'(iy + U_t)| \geq y^{-\theta}]  \leq Cy^b\] then the trace exists.
\end{lemma}

\begin{proof} By using of Borel-Cantelli lemma, it is easy that almost surely for $n$ large enough,  
\[ |f_{k2^{-2n}}'( i 2^{-n} + U_{ k 2^{-2n}})| \leq 2^{ n \theta} \] 
for all $ k = 0, 1, .., 2^{2n} - 1$. Now applying results in section $3$ of \cite{JVL11} (Lemma $3.7$ and distortion Theorem in particular ) completes the proof.  
\end{proof}

\end{document}